\documentclass[12pt]{article}
\usepackage{amssymb,amsmath,epsf, amsthm}
\usepackage{dsfont}
\usepackage{graphicx}

\DeclareMathOperator{\Diff}{Diff} 
\DeclareMathOperator{\dist}{dist} \DeclareMathOperator{\D}{D}
\DeclareMathOperator{\Id}{Id}

\newtheorem{lem}{Lemma }
\newtheorem{stat}{Statement }
\newtheorem{thm}{Theorem }
\newtheorem{conj}{Conjecture }
\theoremstyle{definition}
\newtheorem{defin}{Definition }
\newtheorem{rem}{Remark }

\newcommand{\ep}{\varepsilon}

\newcommand{\lam}{\lambda}
\newcommand{\ZZ}{\mathds{Z}}
\newcommand{\RR}{\mathds{R}}
\newcommand{\sref}[1]{(\ref{#1})}

\begin{document}

\begin{center}{\large\bf Partial hyperbolicity and central shadowing}
\bigskip

SERGEY KRYZHEVICH AND SERGEY TIKHOMIROV

\end{center}

\textbf{Abstract.} We study shadowing property for a partially
hyperbolic diffeomorphism $f$. It is proved that if $f$ is
dynamically coherent then any pseudotrajectory  can be shadowed by a
pseudotrajectory with ``jumps'' along the central foliation. The proof
is based on the Tikhonov-Shauder fixed point theorem.


\section{Introduction}

The theory of shadowing of approximate trajectories
(pseudotrajectories) of dynamical systems is now a well developed
part of the global theory of dynamical systems (see, for example,
monographs \cite{Palmer1}, \cite{Pilyugin}). This theory is of special importance for numerical simulations
and the classical theory of structural stability.

It is well known that a diffeomorphism has the shadowing property in
a neighborhood of a hyperbolic set \cite{Anosov}, \cite{Bowen} and a
structurally stable diffeomorphism has the shadowing property on the
whole manifold \cite{Morimoto}, \cite{Robinson}, \cite{Sawada}.

There are a lot of examples of non-hyperbolic diffeomorphisms, which have shadowing property (see for instance \cite{PilDCDS}, \cite{HolSh})
at the same time this phenomena is not frequent. More precisely the following statements are correct.
Diffeomophisms with $C^1$-robust shadowing property are structurally stable \cite{Sak}.
In \cite{AbdDiaz} Abdenur and Diaz conjectured that $C^1$-generically shadowing is equivalent to structural
stability, and proved this statement for so-called tame diffeomorphisms. Lipschitz shadowing is equivalent to
structural stability \cite{PilTikh} (see \cite{HolSh} for some generalizations).

In present article we study shadowing property for partially hyperbolic diffeomorphisms. Note that due to \cite{Diaz}
one cannot expect that in general shadowing holds for partially hyperbolic diffeomorphisms. We use notion of
central pseudotrajectory and prove that any pseudotrajectory of a partially hyperbolic diffeomorphism can be shadowed by a central pseudotrajectory. This result might be considered as a generalization of a classical shadowing lemma for the case of partially hyperbolic diffeomorphisms.

\section{Definitions and the main result}

Let $M$ be a compact $n$~-- dimensional $C^{\infty}$ smooth manifold,
with a Riemannian metric $\dist$. Let $|\cdot|$ be the Euclidean norm at
${\mathbb R}^n$ and the induced norm on the leaves of the tangent
bundle $TM$.
For any $x\in M$, $\varepsilon>0$ we denote
$$B_\varepsilon(x)=\{y\in M: \dist(x,y)\le \varepsilon\}.$$

Below in the text we use the following definition of partial hyperbolicity (see for example \cite{kewi}).
\begin{defin} A diffeomorphism $f\in \Diff^1(M)$ is called \emph{partially
hyperbolic} if there exists $m\in {\mathbb N}$ such that the mapping
$f^m$ satisfies the following property. There exists a continuous invariant
bundle
$$T_x M=E^s(x) \oplus E^c(x) \oplus E^u(x), \qquad x\in M$$
and continuous positive functions $\nu,\hat\nu, \gamma, \hat\gamma : M \to \RR$ such that
$$\nu,\hat\nu<1, \qquad \nu<\gamma<\hat\gamma<\hat\nu^{-1}
$$
and for all $x\in M$, $v\in{\mathbb R}^n$, $|v|=1$
\begin{equation}\label{Add2.1}
\begin{array}{c}
  |Df^m(x)v|\le \nu(x), \quad v\in E^s(x);\\
\gamma (x)\le |Df^m(x)v|\le \hat\gamma(x), \quad v\in E^c(x);\\
|Df^m(x)v| \geq \hat\nu^{-1}(x), \quad v\in E^u(x).
\end{array}
\end{equation}
\end{defin}

Denote
$$
E^{cs}(x)= E^s(x) \oplus  E^{c}(x), \qquad E^{cu}(x)=E^{c}(x)\oplus E^u(x).
$$

For further considerations we need the notion of dynamical coherence.
\begin{defin} We say that a $k$~-- dimensional distribution $E$ over $TM$ is
\emph{uniquely integrable} if there exists a $k$~-- dimensional continuous
foliation $W$ of the manifold $M$, whose leaves are tangent to $E$
at every point. Also, any $C^1$~-- smooth path tangent to $E$ is
embedded to a unique leaf of $W$.
\end{defin}

\begin{defin} A partially hyperbolic diffeomorphism $f$ is
\emph{dynamically coherent} if both the distributions $E^{cs}$ and
$E^{cu}$ are uniquely integrable.
\end{defin}

If $f$ is dynamically coherent then distribution $E^c$ is also uniquely integrable and
corresponding foliation $W^c$ is a subfoliation of both $W^{cs}$ and $W^{cu}$. For a discussion how often partially hyperbolic
diffeomorphisms are dynamically coherent see \cite{Brin}, \cite{PHSurvey}.

In the text below we always assume that $f$ is dynamically coherent.

For $\tau\in \{s,c,u,cs,cu\}$ and $y \in W^{\tau}(x)$ let $\dist_{\tau}(x,y)$ be the inner distance on $W^{\tau}(x)$ from $x$ to $y$. Note that
\begin{equation}\label{eq1}
\dist(x, y) \leq \dist_{\tau}(x, y), \quad y \in W^{\tau}(x).
\end{equation}
Denote
$$
W^{\tau}_\varepsilon(x) = \{y \in W^{\tau}(x), \; \dist_{\tau}(x, y) < \ep\}.
$$

Let us recall the definition of the shadowing property.

\begin{defin} A sequence $\{x_k:k\in {\mathbb Z}\}$ is called $d$~-
\emph{pseudotrajectory} ($d>0$) if $\dist(f(x_k),x_{k+1}) \leq d$ for
all $k\in {\mathbb Z}$.
\end{defin}

\begin{defin}
Diffeomorphism $f$ satisfies the
\emph{shadowing property} if for any $\ep > 0$ there exists $d> 0$ such that for any $d$-pseudotrajectory $\{x_k: k\in {\mathbb Z}\}$ there exists a
trajectory $\{y_k\}$ of the diffeomorphism $f$ such that
\begin{equation}\label{Add3.1}
\dist(x_k,y_k)\le \ep, \quad k\in {\mathbb
Z}.
\end{equation}
\end{defin}

\begin{defin} Diffeomorphism $f$ satisfies the
\emph{Lipschitz shadowing property} if there exist ${\cal L},
d_0>0$ such that for any $d\in (0, d_0)$, and any $d$-pseudotrajectory $\{x_k: k\in {\mathbb Z}\}$ there exists a
trajectory $\{y_k\}$ of the diffeomorphism $f$, satisfying \sref{Add3.1} with $\ep = {\cal L}d$.
\end{defin}

As was mentioned before in a neighborhood of a hyperbolic set diffeomorphism satisfies the Lipschitz shadowing property \cite{Anosov}, \cite{Bowen}, \cite{Pilyugin}.

We suggest the following generalization of the shadowing property for partially hyperbolic dynamically coherent diffeomorphisms.

\begin{defin}[see for example \cite{hps}] An $\ep$-pseudotrajectory $\{y_k\}$ is called \emph{central} if
for any $k\in {\mathbb Z}$ the inclusion $f(y_k)\in W^c_{\ep}(y_{k+1})$ holds (see Fig.~\ref{CentralPst}).
\end{defin}

\begin{figure}[ht]
        \centering
        \includegraphics[width = 0.5\textwidth]{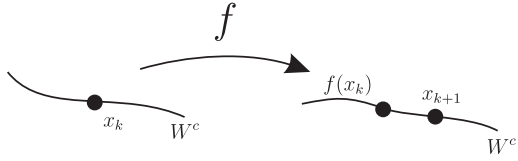}
        \caption{Central pseudotrajectory}
        \label{CentralPst}
\end{figure}

\begin{defin}\label{defCenSh}
A partially hyperbolic dynamically coherent diffeomorphism $f$ satisfies the
\emph{central shadowing property} if for any $\ep > 0$ there exists $d> 0$ such that for any $d$-pseudotrajectory $\{x_k: k\in {\mathbb Z}\}$ there exists an
$\ep$-central pseudotrajectory $\{y_k\}$ of the diffeomorphism $f$, satisfying \sref{Add3.1}.
\end{defin}

\begin{defin}\label{defLipCenSh}
A partially hyperbolic dynamically coherent diffeomorphism
$f$ satisfies the \emph{Lipschitz central shadowing property} if
there exist $d_0, {\cal L}>0$ such that for any $d \in (0, d_0)$
and any $d$-pseudotrajectory $\{x_k: k\in {\mathbb Z}\}$ there
exists an $\ep$-central pseudotrajectory $\{y_k\}$, satisfying
\sref{Add3.1} with $\ep = {\cal L}d$.
\end{defin}

Note that the Lipschitz central shadowing property implies the central shadowing property.

We prove the following analogue of the shadowing lemma for partially hyperbolic
diffeomorphisms.

\begin{thm}\label{thm1} Let diffeomorphism  $f\in C^1$ be partially
hyperbolic and dynamically coherent. Then $f$ satisfies the
Lipschitz central shadowing property.
\end{thm}
Note that for Anosov diffeomorphisms any central pseudotrajectory is
a true trajectory.

Let us also mention the following related notion \cite{hps}.

\begin{defin}
Partially hyperbolic, dynamically coherent diffeomorphism $f$ is called \emph{plaque expansive} if there exists $\ep >0$ such that for any
$\ep$-central pseudotrajectories $\{y_k\}$, $\{z_k\}$, satisfying
$$
\dist(y_k, z_k) < \ep, \quad k \in {\mathbb Z}
$$
hold inclusions
$$
z_0 \in W^c_{\ep}(y_0), \quad k \in {\mathbb Z}.
$$
\end{defin}
In the theory of partially hyperbolic diffeomorphisms the following conjecture plays important role \cite{bdv}, \cite{hps}.
\begin{conj}[Plague Expansivity Conjecture]
Any partially hyperbolic, dynamically coherent diffeomorphism is plaque expansive.
\end{conj}

Let us note that if the diffeomorphism $f$ in Theorem \ref{thm1} is additionally plaque expansive then leaves $W^c(y_k)$ are uniquely defined (see Remark \ref{rem1} below).

Among results related to Theorem \ref{thm1} we would like to mention  that partially hyperbolic dynamically coherent diffeomorphisms, satisfying plaque expansivity property
are leaf stable (see \cite[Chapter 7]{hps}, \cite{PSW} for details).

\section{Proof of Theorem \ref{thm1}}

In what follows below we will use the following statement, which is
consequence of transversality and continuity of foliations $W^s$, $W^{cu}$.
\begin{stat}{\label{st1}}
There exists $\delta_0 > 0$, $L_0 >1$ such that for any $\delta \in (0, \delta_0]$ such that for any $x, y \in
M$ satisfying $\dist(x, y) < \delta$ there exists unique point $z =
W^{s}_{\ep}(x) \cap W^{cu}_{\ep}(y)$ for $\ep = L_0 \delta$.
\end{stat}

Note that for a fixed diffeomorphism $f$, satisfying the assumptions
of the theorem, it suffices to prove that its fixed power $f^m$ satisfies the Lipschitz central shadowing
property. Since foliations $ W^{\tau}$, $\tau \in\{s, u, c, cs, cu\}$ of $f^m$ coincide
with the corresponding foliations of the initial diffeomorphism $f$ we can assume without loss of generality  that conditions \sref{Add2.1} hold for $m=1$.
Note that a similar claim can be done using adapted metric, see \cite{Gour}.

Denote
$$
\lambda = \min_{x\in M}(\min({\hat\nu}^{-1}(x),\nu^{-1}(x)))>1.
$$
Let us choose $l$ so big that
$$
\lambda^l > 2 L_0.
$$
Arguing similarly to previous paragraph it is sufficient to prove that $f^l$ has the Lipschitz central shadowing property and hence,
we can assume without loss of generality that $l = 1$.

Decreasing $\delta_0$ if necessarily we conclude from inequalities \sref{Add2.1} that
\begin{equation}\label{Add5.1}
\dist_s(f(x), f(y)) \leq \frac{1}{\lam}\dist_s(x, y), \quad y \in W^s_{\delta_0}(x)
\end{equation}
and
\begin{equation}\notag
\dist_u(f(x), f(y)) \geq \lam \dist_u(x, y), \quad y \in W^u_{\delta_0}(x).
\end{equation}

Denote
$$
I_r^{\tau}(x)=\{z^{\tau}\in E^{\tau}(x), \; |z^{\tau}|\le r\}, \quad \tau \in\{s, u, c, cs, cu\}, \quad r > 0,
$$
$$
I_r(x)=\{z\in T_xM, \; |z|\le r\}, \quad r > 0.
$$
Consider standard exponential mappings $\exp_x: T_xM \to M$ and
$\exp_x^{\tau}: T_xW^{\tau}(x) \to W^{\tau}(x)$, for $\tau \in
\{s, c, u, cs, cu\}$.
Standard properties of exponential mappings imply that there exists $\ep_0>0$,
such that for all $x\in M$ maps $\exp_x$, $\exp_x^{\tau}$ are well defined on $I_{\ep_0}(x)$ and $I_{\ep_0}^{\tau}(x)$ respectively and
$\D \exp_x(0) = \Id$, $\D
\exp_x^{\tau}(0) = \Id$. Those equalities imply the following.

\begin{stat}{\label{st2}}
For $\mu > 0$ there exists $\ep \in (0, \ep_0)$ such that for any point $x
\in M$, the following holds.
\begin{itemize}
\item[\textbf{A1}] For any $y, z \in B_{\ep}(x)$ and $v_1, v_2 \in I_{\ep}(x)$ the following inequalities hold
$$
\frac{1}{1 + \mu}\dist(y, z) \leq |\exp^{-1}_x(y) - \exp^{-1}_x(z)|
\leq (1 + \mu)\dist(y, z),
$$
$$
\frac{1}{1 + \mu}|v_1 - v_2| \leq \dist(\exp_x(v_1), \exp_x(v_2))
\leq (1+\mu)|v_1 - v_2|.
$$
\item[\textbf{A2}] Conditions similar to \textbf{A1} hold for
$\exp^{\tau}_x$ and $\dist_\tau$, $\tau \in \{s, c, u, cs, cu\}$.
\item[\textbf{A3}] For $y \in W^{\tau}_{\ep}(x)$, $\tau \in \{s, c, u, cs,
cu\}$ the following holds
$$
\dist_{\tau}(x, y) \leq (1+\mu)\dist(x, y).
$$
\item[\textbf{A4}] If $\xi < \ep$ and $y \in W^{cs}_{\xi}(x) \cap W^{cu}_{\xi}(x)$ then
$$
\dist_{c}(x, y) \leq (1+\mu)\xi.
$$
\end{itemize}
\end{stat}

Consider small enough $\mu \in (0, 1)$ satisfying the following inequality
\begin{equation}\label{eq2}
(1+\mu)^2L_0/\lam < 1.
\end{equation}
Choose corresponding $\ep > 0$ from Statement \ref{st2}. Let $\delta = \min(\delta_0, \ep/L_0)$.

For a pseudotrajectory $\{x_k\}$ consider maps $h^s_k: U_k \subset E^s(x_k) \to
E^s(x_{k+1})$ defined as the following:
$$
h_k^s(z) = (\exp^{s}_{x_{k+1}})^{-1} (p)
$$
where
\begin{equation}\label{Add6.1}
p = W^{cu}_{L_0\delta_0}(f(\exp^s_{x_k}(z))) \cap W^s_{L_0\delta_0}(x_{k+1})
\end{equation}
and $U_k$ is the set of points for which map $h^s_k$ is well-defined (see Fig.~\ref{MapHDef}).
Note that maps $h_k^s(z)$ are continuous. The following lemma plays a central role in the proof of
Theorem~\ref{thm1}.

\begin{figure}[ht]
        \centering
        \includegraphics[width = 0.75\textwidth]{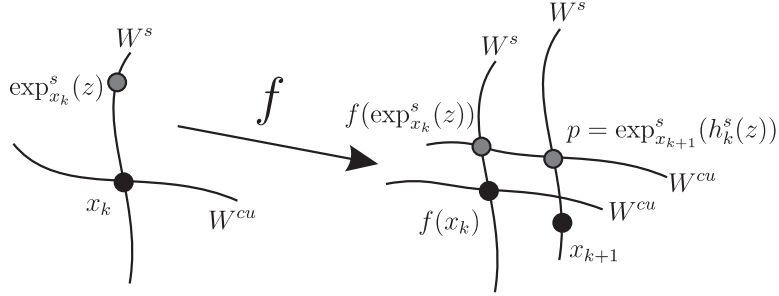}
        \caption{Definition of map $h_k^s$}
        \label{MapHDef}
\end{figure}

\begin{lem}\label{lem1}
There exists $d_0 >0$, $L > 1$ such that for any $d<d_0$ and
$d$-pseudotrajectory $\{x_k\}$ maps $h_k^s$ are well-defined for
$z \in I_{Ld}^s(x_k)$ and the following
inequalities hold
\begin{equation}\label{eq10}
|h_k^s(z)| \leq Ld, \quad k \in \ZZ.
\end{equation}
\end{lem}
\begin{proof}
Inequality \sref{eq2} implies that there exists $L >0$ such that
\begin{equation}\label{eq5}
L_0(1 +L(1+\mu)/\lam)(1+\mu)<L.
\end{equation}
Let us choose $d_0 < \delta_0/2L$. Fix $d < d_0$, $d$-pseudotrajectory
$\{x_k\}$, $k \in \ZZ$ and $z \in I_{Ld}^s(x_k)$.

Condition \textbf{A2} of Statement \ref{st2} implies that
$$
\dist_s(x_k, \exp^s_{x_k}(z)) \leq Ld(1+\mu).
$$
Inequality \sref{Add5.1} implies the following
$$
\dist_s(f(x_k), f(\exp^s_{x_k}(z))) \leq \frac{1}{\lam}Ld(1+\mu).
$$
Inequalities \sref{eq1} and $\dist(f(x_k), x_{k+1}) < d$ imply (see Fig. \ref{FigProof} for illustration)
\begin{multline*}
\dist(x_{k+1}, f(\exp^s_{x_k}(z))) \leq \dist(x_{k+1}, f(x_k)) + \dist(f(x_{k}), f(\exp^s_{x_k}(z))) \leq \\
d\left(1 + \frac{1}{\lam}L(1+\mu)\right) < Ld < \delta_0.
\end{multline*}

\begin{figure}[ht]
        \centering
        \includegraphics[width = 0.60\textwidth]{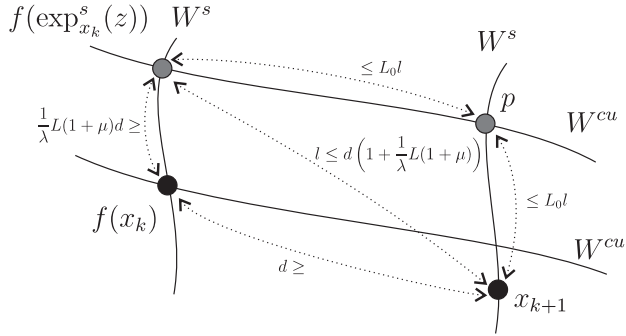}
        \caption{Illustration of the proof of Lemma \ref{lem1}}
        \label{FigProof}
\end{figure}

Statement \ref{st1} implies that point $p$ from relation \sref{Add6.1}
is well-defined and inequality \sref{eq5} implies the following
$$
\dist_s(p, x_{k+1}), \dist_{cu}(p, f(\exp^s_{x_{k}}(z))) < dL_0(1 + \frac{1}{\lam}L(1+\mu)) < \frac{Ld}{1+\mu}.
$$
This inequality and Statement \ref{st2} imply
\begin{equation}\label{eq16}
\dist_{cu}(f(\exp^s_{x_{k+1}}(z)), \exp^s_{x_k}(h_k^s(z))) < Ld,
\end{equation}
$$
|h_k^s(z)| < Ld,
$$
which completes the proof.

\end{proof}

Let $d_0, L> 0$ are constants provided by Lemma \ref{lem1}. Let
$d<d_0$ and $\{x_k\}$ is a $d$-pseudotrajectory.
Denote
$$
X^{s}=\prod_{k=-\infty}^\infty I_{Ld}^{s}(x_k).
$$
This set endowed with the Tikhonov product topology is
compact and convex.

Let us consider map $H: X^s \to X^s$ defined as following
$$
H(\{z_k\}) = \{z'_{k+1}\}, \quad \mbox{where} \quad z'_{k+1} = h_k^s(z_k).
$$
By Lemma \ref{lem1} this map is well-defined. Since $z'_{k+1}$
depends only on $z_k$ map $H$ is continuous.
Due to the Tikhonov-Schauder theorem \cite{tikhsch}, the mapping
$H$ has a (maybe non-unique) fixed point $\{z_k^*\}$. Denote $y_k^s =
\exp_{x_k}^s(z_k^*)$. Since $z_{k+1}^* = h^s_k(z_k^*)$, inequality
\sref{eq16} implies that
\begin{equation}\label{Add7.1}
y^s_{k+1} \in W^{cu}_{Ld}(f(y^s_k)) , \quad k \in \ZZ.
\end{equation}
Since $|z_k^*| < Ld$ we conclude
$$
\dist(x_k, y^s_k) \leq \dist_s(x_k, y^s_k) < (1+\mu)Ld< 2Ld, \quad k \in \ZZ.
$$

Similarly (decreasing $d_0$ and increasing $L$ if necessarily) one may show that there
exists a sequence $\{y^u_k \in W_{2Ld}^u(x_k)\}$ such that
$$
y^u_{k+1} \in W^{cs}_{Ld}(f(y^u_k)), \qquad k \in \ZZ.
$$
Hence $\dist(y^s_k, y^u_k) < \dist(y^s_k, x_k) + \dist(x_k, y^u_k) < 4Ld$. Decreasing $d_0$ if necessarily we can assume that $4L_0Ld < \delta_0$.
Then there exists an unique point $y_k = W^{cu}_{4L_0Ld}(y_k^s) \cap W^{s}_{4L_0Ld}(y_k^u)$ and inclusion \sref{Add7.1} implies that for all $k \in \ZZ$ the following holds
\begin{multline*}
\dist_{cu}(y_{k+1}, f(y_k)) <
\\
\dist_{cu}(y_{k+1}, y_{k+1}^s) + \dist_{cu}(y_{k+1}^s, f(y_k^s)) + \dist_{cu}(f(y_k^s), f(y_k)) < \\
4L_0Ld + Ld + 4RL_0Ld = L_{cu}d,
\end{multline*}
where $R = \sup_{x \in M} |\D f(x)|$ and $L_{cu} > 1$ do not depends on $d$.
Similarly for some constant $L_{cs} > 1$ the following inequalities hold
$$
\dist_{cs}(y_{k+1}, f(y_k)) < L_{cs}d, \quad k \in \ZZ.
$$
Reducing $d_0$ if necessarily we can assume that points $y_{k+1}$, $f(y_k)$ satisfy assumptions of condition \textbf{A4} of Statement \ref{st2}, hence
$$
\dist_c(y_{k+1}, f(y_k)) < (1+\mu)\max(L_{cs}, L_{cu})d, \quad k \in \ZZ
$$
and sequence $\{y_{k}\}$ is an $L_1d$-central pseudotrajectory with $$L_1 = (1+\mu)\max(L_{cs}, L_{cu}).$$

To complete the proof let us note that
$$
\dist(x_k, y_k)< \dist(x_k, y_k^s) + \dist(y_k^s, y_k) < 2Ld + 4L_0Ld, \quad k \in \ZZ.
$$

Taking ${\cal L} = \max(L_1, 2L + 4L_0)$ we conclude that $\{y_k\}$ is an
${\cal L}d$-central pseudotrajectory which ${\cal L}d$ shadows
$\{x_k\}$. $\square$

\begin{rem}\label{rem1}
Note that we do not claim uniqueness of such sequences $\{y_k^s\}$ and $\{y_k^u\}$.
In fact it is easy to show (we leave details to the reader)
that uniqueness of those sequences is equivalent to the plaque expansivity conjecture.
\end{rem}

\section{Acknowledgement} Sergey Kryzhevich was supported by the UK
Royal Society (joint project with Aberdeen University), by the Russian Federal Program "Scientific and
pedagogical cadres", grant no.  2010-1.1-111-128-033. Sergey
Tikhomirov was supported by the Humboldt postdoctoral fellowship for
postdoctoral researchers (Germany). Both the coauthors are grateful
to the Chebyshev Laboratory (Department of Mathematics and
Mechanics, Saint-Petersburg State University) for the support under
the grant 11.G34.31.0026 of the Government of the Russian
Federation.

\end{document}